\documentclass[11pt]{amsart}
\usepackage{comment}
\usepackage{color,amssymb,amsmath,amsthm,mathrsfs,amstext,amssymb}
\usepackage{tabularx}
\usepackage{array}
\usepackage[square,sort,comma,numbers]{natbib}
\usepackage{url} 
\usepackage{hyperref}
\usepackage{accents}

\theoremstyle{plain}
\newtheorem{lemma}{Lemma}[section]
\newtheorem{prp}[lemma]{Proposition}
\newtheorem{thm}[lemma]{Theorem}
\newtheorem*{thm*}{Theorem}
\newtheorem{crl}[lemma]{Corollary}
\newtheorem*{crl*}{Corollary}
\newtheorem*{definition*}{Definition}
\newtheorem*{clm*}{Claim}
\newtheorem{question}[lemma]{Question}
\newtheorem{claim}[lemma]{Claim}
\newtheorem{dfn}[lemma]{Definition}

\newtheorem{remark}{Remark}

\newcommand{\om}{\omega}
\newcommand{\ZFC}{\mathsf{ZFC}}
\newcommand{\CH}{\mathsf{CH}}
\newcommand{\forces}{\Vdash}
\newcommand{\F}{\mathcal F}
\newcommand{\w}{\omega}
\newcommand{\B}{\mathcal B}

\newcommand{\C}{\mathcal C}
\newcommand{\J}{\mathcal J}
\usepackage{comment}
\newcommand{\A}{\mathcal A}

\makeatother 

\usepackage{babel}

\title{Filters, ideal independence and ideal Mr\'owka spaces}

\author{S. Bardyla}
\address{Institute of Mathematics, Univeristy of Vienna, Kolingasse 14-16, 1090 Vienna, Austria}
\email{sbardyla@gmail.com}

\author{J. Cancino-Manr\'iquez}
\address{Institute of Mathematics, Czech Academy of Sciences, \v{Z}itn\'a 25, Praha 1,
Praha, Czech Republic}
\email{cancino@math.cas.cz}

\author{V. Fischer}
\address{Institute of Mathematics, Univeristy of Vienna, Kolingasse 14-16, 1090 Vienna, Austria}
\email{vera.fischer@univie.ac.at}

\author{C. Bacal Switzer}
\address{Institute of Mathematics, Univeristy of Vienna, Kolingasse 14-16, 1090 Vienna, Austria}
\email{corey.bacal.switzer@univie.ac.at}

\thanks{\emph{Acknowledgments.}: The first author was supported by the Austrian Science Fund FWF (Grant  ESP 399). The second, third and fourth authors would like to thank the Austrian Science Fund (FWF) for the generous support through START Grant Y1012-N35. 
}



\begin{document}

\begin{abstract}
A family $\mathcal{A} \subseteq [\omega]^\omega$ such that for all finite $\{X_i\}_{i\in n}\subseteq \mathcal A$ and $A \in \mathcal{A} \setminus \{X_i\}_{i\in n}$, the set $A \setminus \bigcup_{i \in n} X_i$ is infinite, is said to be ideal independent. 

We prove that an ideal independent family $\A$ is maximal if and only if $\A$ is $\J$-completely separable and maximal $\J$-almost disjoint for a particular ideal $\J$ on $\omega$. We show that $\mathfrak{u}\leq\mathfrak{s}_{mm}$, where $\mathfrak{s}_{mm}$ is the minimal cardinality of maximal ideal independent family. This, in particular, establishes the independence of $\mathfrak{s}_{mm}$ and $\mathfrak{i}$. Given an arbitrary set  $C$ of uncountable cardinals, we show how to simultaneously adjoin via forcing maximal ideal independent families of  cardinality $\lambda$ for each $\lambda\in C$, thus establishing the consistency of $C\subseteq \hbox{spec}(\mathfrak{s}_{mm})$. Assuming $\CH$, we construct  a maximal ideal independent family, which remains maximal after forcing with any proper, $^\omega\omega$-bounding, $p$-point preserving forcing notion and evaluate $\mathfrak{s}_{mm}$ in several well studied forcing extensions.

We also study natural filters associated with ideal independence and introduce an analog of Mr\'owka spaces for ideal independent families.
\end{abstract}

\maketitle

\section{Introduction}

Given a family $\mathcal A \subseteq [\omega]^\omega$, the ideal generated by $\A$ is the collection of all $X \subseteq \omega$ so that $X \subseteq^* \bigcup_{i < n} A_i$ for some $\{A_i\}_{i\in n}\subseteq \mathcal A$ where $\subseteq^*$ means inclusion mod finite. A family $\mathcal{A}$ is {\em ideal independent} if no $A \in \mathcal{A}$ is in the ideal generated by $\mathcal{A} \setminus \{A\}$. 
Almost disjoint families and independent families are both examples of ideal independent families. An easy application of Zorn's lemma shows that there are maximal ideal independent families, however it is not always the case that a maximal almost disjoint or a maximal independent family is maximal ideal independent. In this paper to each ideal independent family $\A$ we correspond an ideal $\J(\A)$ (see Definition~\ref{new}) which allows us to characterize maximal ideal independent families as follows (see Theorem~\ref{char}):

\begin{thm*}
 An ideal independent family $\A$ is maximal if and only if $\A$ is maximal $\J(\A)$-
almost disjoint and $\J(\A)$-completely separable.   
\end{thm*}

By $\mathfrak{s}_{mm}$ we denote the least cardinality of a maximal ideal independent family.
An earlier investigation of $\mathfrak{s}_{mm}$ can be found \cite{cancino_guzman_miller_2021}, where it is shown that ${\rm max}\{\mathfrak{d}, \mathfrak{r}\}\leq\mathfrak{s}_{mm}$ and that each of the following inequalities $\mathfrak{u} < \mathfrak{s}_{mm}$, $\mathfrak{s}_{mm} < \mathfrak{i}$, $\mathfrak{s}_{mm} < \mathfrak{c}$ is consistent. Here $\mathfrak{d}$, $\mathfrak{u}$, $\mathfrak{r}$,  $\mathfrak{i}$ denote the dominating number, the ultrafilter number, the reaping and  independence numbers, respectively.  We refer the reader to \cite{blass_handbook} for  definitions and basic properties of the combinatorial cardinal characteristics, which are not stated here. Strengthening and complementing the above results, in Section 2, we establish the following $\ZFC$ inequality, which also answers  Question 17 of \cite{cancino_guzman_miller_2021}, see Theorem \ref{mainthm1}:

\begin{thm*}
$\mathfrak{u}\leq \mathfrak{s}_{mm}$. 
\end{thm*}

Consequently, we obtain the independence of $\mathfrak{s}_{mm}$ and $\mathfrak{i}$, as the consistency of $\mathfrak{s}_{mm} < \mathfrak{i}$ is shown in \cite[Theorem 16]{cancino_guzman_miller_2021}, while the consistency of $\mathfrak{i} < \mathfrak{u}$ is established in Shelah's \cite{con_i_u} and hence by the above theorem, $\mathfrak{i} < \mathfrak{s}_{mm}$ holds in the latter model.

\begin{crl*}
$\mathfrak{s}_{mm}$ and $\mathfrak{i}$ are independent.
\end{crl*}

A key role in our investigations is taken by specific filters, which are naturally associated to a given ideal independent family. On one side, these are filters to which we refer as {\em complemented filters}, see Definition \ref{new} and on the other side, filters resembling the notion of a diagonalization filter for an independent family, see for example \cite[Definition 1]{VFSS1}. In difference with earlier instances of diagonalization reals, associated to say almost disjoint families, towers, or cofinitary groups, the existence of a diagonalization
real for a given ideal independent family, employs a Cohen real (see Lemma \ref{extension_lemma}). Adjoining diagonalization reals for ideal independent families cofinally along an appropriate finite support iteration, as well as building on and modifying earlier forcing constructions used to  control for example the spectrum of independence (see in particular \cite{VFSS1, VFSS2}) we establish the following (see Theorem \ref{mainthm2}):


\begin{thm*} 
(GCH) Let $R$ be a set of regular uncountable cardinals. Then, there is a ccc generic extension in which for every $\lambda\in R$ there is a maximal ideal independent family of cardinality $\lambda$. 
\end{thm*}

Moreover, we look at the preservation of small witnesses to $\mathfrak{s}_{mm}$. The preservation of the maximality of extremal sets  of reals, like mad families, maximal eventually different families of reals, or maximal independent families, under forcing iterations is usually a non-trivial task and often involves the construction of a combinatorial object which is maximal in a strong sense, examples given by tight almost disjoint and selective independent families. Partially inspired by the notion of a $\mathcal{U}$-supported maximal independent family in the higher Baire spaces given in \cite{VFDM}, in Definition \ref{encompassing} we introduce the notion of an {\emph{$\mathcal{U}$-encompassing ideal independent family}} and establish the following powerful preservation result (see Theorems \ref{mainthm3} and \ref{preserving_encompassing}). 

\begin{thm*}
(CH) There is a maximal ideal independent family $\mathcal{A}$ which remains maximal, and so a witness to  $\mathfrak{s}_{mm} = \aleph_1$, in any generic extension obtained by a proper, $\om^\om$-bounding, $p$-points preserving forcing notion.
\end{thm*}

The above theorem applies to a large class of partial orders and implies that in many well-studied forcing extensions, $\mathfrak{s}_{mm}=\max\{\mathfrak{d}, \mathfrak{u}\}$. 

Finally, to each ideal independent family $\A$ we correspond a topological space $\psi(\A)$ which is a generalized version of the well-known Mr\'{o}wka space. Mr\'{o}wka spaces play important role in set-theoretic topology (see \cite{hrusak-separable-mad} and references therein). In this paper we characterize ideal independent families whose corresponding spaces are locally compact or pseudocompact (see Propositions~\ref{top1} and~\ref{top2}). Suprisingly this corresponds to the case of almost disjoint families, thus suggesting that Mr\'{o}wka spaces of non-almost disjoint families behave very differently to the classical theory. Specifically we show the following.

\begin{thm*}
    For a maximal ideal independent family $\A$ the following are equivalent.
    \begin{enumerate}
        \item $\psi(\A)$ is pseudocompact.
        \item $\A$ is  MAD family.
        \item $\A$ is a completely separable MAD family.
    \end{enumerate}
\end{thm*}


We conclude the paper with a brief discussion of remaining open questions.


\section{Preliminaries}
Recall that given $f, g\in \om$ we write $f \leq^* g$  ($f$ is {\em eventually dominated by} $g$), provided there is $n\in\omega$ such that for all $m\geq n$, $f(m)\leq g(m)$. The cardinal $\mathfrak{b}$, the {\em bounding number} is the least size of a family $A \subseteq \om^\om$ so that no single $f \in \om^\om$ eventually dominates every $g \in A$. Dually, the {\em dominating number}, $\mathfrak{d}$, is the least size of a {\em dominating family}, that is a family $D \subseteq\om^\om$ so that every $f \in \om^\om$ is eventually dominated by some $g \in D$. We denote by $[\omega]^\om$ the {\em Ramsey space}, that is the Polish space of infinite subsets of $\om$. Often it is convenient to quotient this space by the ideal $\mathsf{Fin}$ consisting of finite subsets of $\om$. For instance, if $A, B$ in $[\om]^\om$ then we write $A \subseteq^* B$, read ``$A$ is {\em almost contained in} $B$" if $A \setminus B$ is finite. Similarly we say that $A$ and $B$ are {\em almost equal}, denoted $A =^* B$, if their symmetric difference is finite, and we say that $A$ and $B$ are {\em almost disjoint}, denoted $A \cap B =^* \emptyset$, if their intersection is finite. 

Given a family $\A$ of subsets of $\w$ by $\A^+$ we denote the collection of all subsets of $\w$ which are not in the ideal generated by $\A$.
For an ideal $\J$ on $\omega$ which contains the ideal $\mathsf{Fin}$ and subsets $A,B$ of $\omega$ we write $A\subseteq^{\J}B$ if $A\setminus B\in \J$.   A family $\A\subseteq \J^+$ is called {\em $\J$-almost disjoint} if $A_1\cap A_2\in\J$ for any distinct $A_1,A_2\in\A$.  
Given a family $\B\subseteq[\omega]^{\omega}\setminus \J$, let 
$$\B^{+_{\J}}=[\omega]^{\omega}\setminus \{X\in[\w]^{\w}: X\subseteq^\J\bigcup \mathcal Y,\hbox{ where }\mathcal Y\in [\B]^{<\w}\}$$ and   $$\B^{++_\J}=\{X\in [\omega]^{\omega}:\exists \{B_i:i\in\omega\}\in[\B]^{\omega}\hbox{ such that } X\cap B_i\notin \J \quad\forall i\in\w\}.$$ 
Elements of the set $\B^{+_\J}$ are called {\em positive} with respect to the ideal $\J$. If $\J=\mathsf{Fin}$, then we write $\B^+$ ($\B^{++}$, resp.) instead of $\B^{+_\J}$ ($\B^{++_\J}$, resp.).

Given a family $\mathcal A \subseteq [\om]^\om$ we say that $\mathcal A$ has the {\em finite intersection property} if any finite subfamily has infinite intersection.
An infinite subset $B\subseteq \w$ is called a {\em pseudo-intersection} of a family $\A\subseteq [\w]^{\w}$ if $B\subseteq^* A$ for every $A\in\A$. 
The cardinal characteristic $\mathfrak{p}$ is the least size of a family with the finite intersection property with no pseudo-intersection. An ultrafilter is said to be {\em principal} if it contains a singleton and {\em non-principal} otherwise. Unless otherwise stated we will assume all ultrafilters to be non-principal. If $\mathscr{U}$ is an ultrafilter then a {\em base} for $\mathscr{U}$ is a subset $\mathcal B \subseteq \mathscr{U}$ so that every element $A \in \mathscr{U}$ almost contains some $B \in \mathcal{B}$. In this case, we say that $\mathcal B$ {\em generates} $\mathscr{U}$, sometimes denoted $\langle \mathcal B \rangle$. The {\em ultrafilter number} $\mathfrak{u}$, is the least size of a base of a non-principal ultrafilter. 
For an ultrafilter  $\mathscr{U}$ we say that
\begin{enumerate}
    \item $\mathscr{U}$ is a {\em $p$-point} if every countable subfamily of $\mathscr{U}$ has a pseudo-intersection in $\mathscr{U}$.
    \item $\mathscr{U}$ is a {\em q-point} if every partition of $\om$ into finite sets $\{I_n\}_{n < \om}$ there is a $U \in \mathscr{U}$ so that $|U \cap I_n| = 1$ for each $n < \om$.
    \item $\mathscr{U}$ is {\em Ramsey}, or, {\em selective} if it is a $p$-point and a q-point.
    \item $\mathscr{U}$ is a $p_{\mathfrak{c}}$-point if any  
 $\mathscr{F}\subseteq\mathscr{U}$, $|\mathscr{F}|<\mathfrak{c}$ has a pseudo-intersection in $\mathscr{U}$.
\end{enumerate}

Finally, a family $\mathcal I \subseteq [\om]^\om$ is said to be {\em independent} if whenever $\mathcal A, \mathcal B$ are finite, disjoint, non-empty subfamilies of $\mathcal I$, the set $\bigcap A \setminus \bigcup \mathcal B$ is infinite. While this notion technically makes sense for finite families it is somewhat degenerate in this case thus we will implicitly mean by an independent family an {\em infinite} independent family. The least size of a maximal (infinite) independent family is denoted $\mathfrak{i}$.


\section{Filters, Almost Disjointness and Ideal Independence}

The following definition plays a central role in this section.

 \begin{dfn}\label{new}
Let $\A$ be an ideal independent family and $A\in\A$. 
\begin{enumerate}
    \item By $\F(\A,A)$ we denote the filter generated by the family $$\{A\setminus\bigcup \mathcal B: A\notin\mathcal B\in [\mathcal{A}]^{<\omega}\}.$$
    \item By $\J(\A,A)$ we denote the ideal generated by the family $$\{A\cap B: B\in \A\setminus\{A\}\}\cup\mathsf{Fin}.$$
    \item By $\J(\A)$ we denote the ideal generated by the family $\bigcup_{A\in\A}\J(\A,A)$.
\end{enumerate} 
\end{dfn}  

We refer to the filters of the form $\mathcal{F}(\mathcal{A}, A)$ as the {\em complemented filters of} $\mathcal{A}$, while for a fixed $A \in \mathcal{A}$ we say that  $\mathcal{F}(\mathcal{A}, A)$ is the {\em complemented filter (of} $\mathcal{A}$) {\em corresponding to} $A$. Observe that for any $A\in\mathcal A$ a subset $B\subseteq A$ belongs to the ideal $\mathcal{J}(\A,A)$ if and only if $A\setminus B\in \F(\A,A)$.  

Note that an ideal independent family $\mathcal{A}$ is maximal if and only if every $X \in [\omega]^\omega$ is either in the ideal generated by $\mathcal{A}$ or belongs to at least one of the filters $\mathcal{F}(\mathcal{A}, A)$ (these two possibilities are not mutually exclusive). The name ``complemented" comes from this observation: under maximality every element of the complement of the ideal generated by $\mathcal{A}$ is in some complemented filter.

The proof of the following lemma is straightforward.

\begin{lemma}\label{easy}
For a family $\A\subseteq [\w]^{\w}$ the following statements hold:
\begin{enumerate}
    \item if for some ideal $\J$ on $\w$ the family $\A$ is $\J$-almost disjoint, then $\A$ is ideal independent; 
    \item if the family $\A$ is ideal independent, then $\A$ is $\J(\A)$-almost disjoint.
\end{enumerate}
\end{lemma}

\begin{definition*}
A $\J$-almost disjoint family $\A$ is called {\em $\J$-completely separable} if for any $B\in \A^{++_{\J}}$ there exists $A\in\A$ such that $A\subseteq^{\J} B$. 
\end{definition*}

One can easily check that $\A^{+_\J}=\A^{++_\J}$ for any maximal $\J$-almost disjoint family $\A$. 
In case $\J=\mathsf{Fin}$, we refer to $\J$-completely separable family simply as completely separable.
The existence of a completely separable maximal almost disjoint family in ZFC is still an open problem (see~\cite{HS} for the history thereof). Completely separable maximal almost disjoint families are intensively studied in the literature (see the survey~\cite{G} and references therein). $\J$-almost disjoint families were recently investigated in~\cite{RS}. The following theorem establishes a connection between $\J$-completely separable $\J$-almost disjoint families and maximal ideal independent families.

\begin{thm}\label{char}
An ideal independent family $\A$ is maximal if and only if $\A$ is maximal $\J(\A)$-almost disjoint and $\J(\A)$-completely separable. 
\end{thm}

\begin{proof}
Let $\A$ be a maximal ideal independent family. By Lemma~\ref{easy}, the family $\A$ is $\J(\A)$-almost disjoint. 
To derive a contradiction assume that there exists $B\in \J(\A)^+\setminus\A$ such that $B\cap A\in \J(\A)$ for any $A\in \A$. Then $\A\cup\{B\}$ is ideal independent, which contradicts the maximality of $\A$. Thus, the family $\A$ is maximal $\J(\A)$-almost disjoint.
To show that the family $\A$ is $\J(\A)$-completely separable fix any $C\in \A^{++_{\J(\A)}}$. Note that $C\in \A^{+}\setminus \A$. By the maximality of $\A$ there exist $A\in\A$ and $\{A_i:i\in n\}\in[\A\setminus\{A\}]^{<\w}$ such that $A\subseteq^* \bigcup_{i\in n}A_i\cup C$. Hence 
$$A=A\cap A\subseteq^* A\cap(\bigcup_{i\in n}A_i\cup C)\subseteq\big(\bigcup_{i\in n}(A\cap A_i)\big)\cup C.$$
Then $A\subseteq^{\J(\A)}C$, as $\bigcup_{i\in n}(A\cap A_i)\in\J(\A,A)\subseteq \J(\A)$. It follows that the family $\A$ is $\J(\A)$-completely separable.

Assume that the family $\A$ is  maximal $\J(\A)$-almost disjoint and $\J(\A)$-completely separable. By Lemma~\ref{easy}, the family $\A$ is ideal independent. Fix any $C\in \A^+$. Let us show that $C\in \A^{+_{\J(\A)}}$. Assuming the contrary we can find $\{A_i: i\in n\}\in[\A]^{<\w}$ and $D\in \J(\A)$ such that $C\subseteq \bigcup_{i\in n}A_i\cup D$. Then there exists $\{D_i:i\in m\}\in [\A]^{<\w}$ such that $D\subseteq^* \bigcup_{i\in m}D_i$. But then $C\subseteq^* \bigcup_{i\in n}A_i\cup(\bigcup_{i\in m}D_i)$, witnessing that $C\notin \A^+$. The obtained contradiction implies that $C\in \A^{+_{\J(\A)}}$.
Taking into account that $\A^{+_{\J(\A)}}=\A^{++_{\J(\A)}}$, the $\J(\A)$-complete separability implies the existence of $A\in\A$ such that $A\subseteq^{\J(\A)}C$. By the definition of $\J(\A)$, the set $A\setminus C$ belongs to $\J(\A)$ and thus it can be covered mod finite by the union of finitely many elements $E_1,\ldots, E_m\in \A\setminus\{A\}$, witnessing that $C\in\F(\A,A)$. Hence the family $\A$ is maximal ideal independent. 
\end{proof}

The following corollary of Theorem~\ref{char} describes almost disjoint families which are maximal ideal independent.

\begin{crl}\label{csmad}
An almost disjoint family $\A$ is maximal ideal independent if and only if $\A$ is a completely separable maximal almost disjoint family.
\end{crl}

Note that, independent families are never maximal ideal independent. Indeed, for every element $A$ of  an independent family $\A$ the family $\A\cup\{\w\setminus A\}$ is ideal independent.

For a filter $\F$ let $\chi(\F)=\min\{|\mathcal B|:\mathcal B$ is a base of $\F\}$.

\begin{thm}\label{sol}
Let $\A=\{A_{\alpha}:\alpha\in\kappa\}$ be a (maximal) ideal independent family and for any $\alpha\in\kappa$, $\mathcal G_\alpha$ be a filter on $\omega$ such that $\F(\A, A_{\alpha})\subseteq \mathcal G_{\alpha}$ and $\chi(\mathcal G_{\alpha})\leq \kappa$. Then there exists a (maximal) ideal independent family $\A'=\{A'_{\alpha}:\alpha\in\kappa\}$ such that $\F(\A',A'_{\alpha})=\mathcal G_{\alpha}$ for all $\alpha\in\kappa$.
\end{thm}

\begin{proof}
Fix an ideal independent family $\A=\{A_{\alpha}:\alpha\in\kappa\}$ and filters  $\mathcal G_{\alpha}$, $\alpha\in\kappa$ which satisfy the conditions above.   
For each $\alpha\in\kappa$ let $\C_{\alpha}$ be a base of $\mathcal G_\alpha$ of cardinality $\leq \kappa$. Moreover assume that for every $\alpha\in\kappa$ and $C\in\C_{\alpha}$ we have that $C\subseteq A_{\alpha}$. 
The latter assumption implies that $\C_{\alpha}\cap \C_{\beta}=\emptyset$ for every $\alpha\neq \beta$.
Enumerate the set  $\mathcal D=\bigcup_{\alpha\in\kappa}\C_{\alpha}$ as $\{D_\alpha:\alpha\in\kappa\}$, additionally assuming that $D_{\alpha}\notin \C_{\alpha}$ for every $\alpha\in\kappa$.
For each $\alpha\in\kappa$ let $A'_\alpha=A_\alpha\cup (A_{\sigma(\alpha)}\setminus D_\alpha)$, where $D_{\alpha}\in\C_{\sigma(\alpha)}$. Since for every $\alpha\neq \beta$ the set $\C_{\alpha}\cap \C_{\beta}$ is empty,  the ordinal $\sigma(\alpha)$ is unique and, thus, well-defined. Let us show that the family $\A'=\{A_{\alpha}':\alpha\in\kappa\}$ is ideal independent. 
Fix any $\alpha\in\kappa$ and $\{\beta_i:i\in n\}\in [\kappa\setminus\{\alpha\}]^{<\omega}$. We need to show that $A_{\alpha}'\setminus\bigcup_{i\in n}A_{\beta_i}'\neq^* \emptyset$. There are two cases:
\begin{enumerate}
    \item $\alpha\notin \{\sigma(\beta_i):i\in n\}$;
    \item $\alpha\in \{\sigma(\beta_i):i\in n\}$. 
\end{enumerate}
    
{\emph{Case 1)}}
Observe that 
$A_{\alpha}\subseteq A_{\alpha}'$ and $$\bigcup_{i\in n}A_{\beta_i}'=\bigcup_{i\in n}(A_{\beta_i}\cup(A_{\sigma(\beta_i)}\setminus D_{\beta_i}))\subseteq \bigcup_{i\in n}(A_{\beta_i}\cup A_{\sigma(\beta_i)}).$$
Therefore, $A_{\alpha}'\setminus \bigcup_{i\in n}A_{\beta_i}'\subseteq A_{\alpha}\setminus \bigcup_{i\in n}(A_{\beta_i}\cup A_{\sigma(\beta_i)})\neq^*\emptyset$, as $\alpha$ does not belong to 
$\{\beta_i:i\in n\}\cup\{\sigma(\beta_i):i\in n\}$ and the family $\A$ is ideal independent.

{\emph{Case 2)}} Let $\Gamma=\{i\in n: \alpha=\sigma(\beta_i)\}$. Since $A_{\alpha}\cap\bigcap_{i\in \Gamma}D_{\beta_i}\in \mathcal G_{\alpha}$ we obtain:  
\begin{align*}
& A_{\alpha}'\setminus \bigcup_{i\in n}(A_{\beta_i}\cup(A_{\sigma(\beta_i)}\setminus  D_{\beta_i})) \supseteq \\ &  \big(A_{\alpha}\setminus (\bigcup_{i\in n} A_{\beta_i})\big) \cap(A_{\alpha}\cap\bigcap_{i\in \Gamma}D_{\beta_i})\cap (A_{\alpha}\setminus (\bigcup_{i\notin \Gamma} A_{\sigma(\beta_i)}))\in \mathcal G_{\alpha}.
\end{align*}
Hence $A_{\alpha}'\setminus\bigcup_{i\in n}A_{\beta_i}'\neq^* \emptyset$, witnessing that the family $\A_{\alpha}'$ is ideal independent.

Fix any $\alpha\in\kappa$ and $C\in\C_\alpha$. There is $\xi\in\kappa\setminus\{\alpha\}$ such that $C=D_\xi\subseteq A_{\xi}$. Observe that $\sigma(\xi)=\alpha$. Then  
\begin{align*}
A_{\alpha}'\setminus (A'_{\sigma(\alpha)}\cup A'_{\xi}) & \subseteq(A_{\alpha}\cup(A_{\sigma(\alpha)}\setminus D_{\alpha}))\setminus \big(A_{\sigma(\alpha)}\cup (A_{\sigma(\xi)}\setminus D_\xi)\big) \\
& \subseteq A_{\alpha}\setminus (A_{\alpha}\setminus D_\xi)= C,
\end{align*}
witnessing that $\mathcal G_{\alpha}\subseteq\F(\A',A'_{\alpha})$. To show the converse inclusion fix any $F\in\F(\A',A'_{\alpha})$. Then there exists
$\{\alpha_i:i\in n\}\in[\kappa\setminus\{\alpha\}]^{<\w}$ such that
$$A_{\alpha}'\setminus \bigcup_{i\in n}A_{\alpha_i}'=\big(A_{\alpha}\cup(A_{\sigma(\alpha)}\setminus D_{\alpha})\big)\setminus\bigcup_{i\in n}\big(A_{\alpha_i}\cup(A_{\sigma(\alpha_i)}\setminus D_{\alpha_i})\big)\subseteq^* F.$$
It is straightforward to check that if $\alpha$ is  not in $\{\sigma(\alpha_i):i\in n\}$, then $$A_{\alpha}'\setminus \bigcup_{i\in n}A_{\alpha_i}'\in \F(\A,A_{\alpha})\subseteq \mathcal G_{\alpha},$$ implying that $F\in \mathcal G_{\alpha}$. Assume that the set $\Gamma=\{i\in n: \sigma(\alpha_i)=\alpha\}$ is not empty. Then $$F\supseteq^* A_{\alpha}'\setminus \bigcup_{i\in n}A_{\alpha_i}'\supseteq \big(A_{\alpha}\setminus \bigcup_{i\in n}A_{\alpha_i}\big)\cap (A_{\alpha}\cap \bigcap_{i\in \Gamma} D_{\alpha_i})\cap  (A_{\alpha}\setminus \bigcup_{i\notin \Gamma}A_{\sigma(\alpha_i)}     \big)\in \mathcal G_{\alpha}.$$
Hence $\mathcal G_{\alpha}=\F(\A',B_{\alpha}')$ for every $\alpha\in\kappa$.


Assume the ideal independent family $\A$ is maximal and fix any $C\in\A'^+$. It is easy to see that $\A^+=\A'^{+}$. By maximality of $\A$ there is $\alpha\in \kappa$ such that $C\in\F(\A,A_{\alpha})\subseteq\mathcal G_{\alpha}=\F(\A',A_{\alpha}')$, witnessing that the family $\A'$ is maximal as well.
\end{proof}

In particular we obtain:

\begin{crl}\label{ultraf}
There exists a maximal ideal independent family $\A$ such that $\F(\A,A)$ is an ultrafilter for all $A\in\A$. 
\end{crl}

\begin{proof}
Let $\mathcal D$ be an almost disjoint family of size $\mathfrak c$. Enlarge it to a maximal ideal independent family $\A$. For any $A\in \A$ fix any ultrafilter $\mathcal G_A$ which contains the filter $\F(\A,A)$. Theorem~\ref{sol} implies that there exists a maximal ideal independent family $\A'$ such that for any $A'\in\A'$ the filter $\F(\A',A')$ coincides with the ultrafilter $\mathcal G_{A}$.  
\end{proof}

We want to point out, that under the assumption $\mathfrak{p}=\mathfrak{c}$ one can obtain in addition the existence of a maximal ideal independent family $\mathcal{A}$ with the property that all complemented filters of $\mathcal{A}$ are  $p_\mathfrak{c}$-points.

The relation between ideal independent families and ultrafilters can be formalized by the following $\ZFC$ theorem. 

\begin{thm}\label{mainthm1}
$\mathfrak{u}\leq\mathfrak{s}_{mm}$.
\end{thm}

\begin{proof}
Assume otherwise $\mathfrak{s}_{mm}<\mathfrak{u}$ and let $\mathcal{A}$ be a maximal ideal independent family of minimal cardinality. Maximality implies that there is a countable subfamily $\{ A_n\}_{n\in\omega}$ of $\mathcal{A}$ whose union is almost equal to $\omega$. Define $B_0=A_0$ and for $n>0$, $B_n=A_n\setminus\bigcup_{i<n}A_i$. For each $n\in\omega$, let $\mathcal{F}_n$ be the filter $\mathcal{F}(\mathcal{A},A_n)\upharpoonright B_n$. Since $\mathcal{F}_n$ is not an ultrafilter (by assumption on the size of $\mathfrak{u}$), for any $s\in 2^{n}$ there is $D_s\in\mathcal{F}_n^+\setminus \mathcal{F}_n$ such that for different $s,r\in 2^n$, $D_r\cap D_s=\emptyset$. Now, for each $f\in 2^\omega$, define $D^f=\bigcup_{n\in\omega}D_{f\upharpoonright n}$. 

\begin{clm*} For each $f\in 2^\omega$ and each $\mathcal F\in[\mathcal A]^{<\omega}$, $D^f\not\subseteq\bigcup \mathcal F$.
\end{clm*}
\begin{proof} Let $\mathcal F\in[\mathcal{A}]^{<\omega}$ and let $n\in\omega$ be such that $A_n\notin \mathcal F$. We can assume that $\{A_i:i<n\}\subseteq \mathcal F$. Since $D_{f\upharpoonright n}$ is $\mathcal{F}_n$-positive, it can not be covered by $\bigcup F$ and so $D^f$ can not be covered by $\bigcup \mathcal F$ either. 
\end{proof}	

By maximality of $\mathcal{A}$, for any $f\in 2^\omega$, there are 
$$A_f\in\mathcal{A}\hbox{ and }\mathcal{F}_f\in[\mathcal{A}\setminus\{A_f\}]^{<\omega}$$ such that $A_f\setminus \bigcup \mathcal{F}_f\subseteq^* D^f$. Note that for no $n\in\omega$ we have  $A_{n}=A_f$, as otherwise, for some $n\in\omega$ we would have $A_{n}\setminus \mathcal{F}_f\subseteq^* D^f$, which implies $$B_{n}\setminus \bigcup \mathcal{F}_f=A_{n}\setminus\left(\bigcup \mathcal{F}_f\cup\bigcup_{i< n}A_i\right)\subseteq^* D_{f\upharpoonright n},$$ contradicting the choice of the set $D_{f\upharpoonright n}$. Since $\mathfrak{s}_{mm}<\mathfrak{c}$, there are different $f,g\in 2^\omega$ such that $A_f=A_g$ and $\mathcal{F}_f=\mathcal{F}_g$. By construction, we have  $$A_f\setminus \mathcal{F}_f\subseteq^* D^f\cap D^g\subseteq\bigcup_{i\leq n_0} B_i,$$ where $n_0\in\omega$ is the maximal natural number such that $f\upharpoonright n_0=g\upharpoonright n_0$.  But $\bigcup_{i\leq n_0}B_i=\bigcup_{i\leq n_0}A_{i}$, which means that $A_f\setminus \bigcup \mathcal{F}_f\subseteq^* \bigcup_{i\leq n_0}A_{n_0}$, and so $A_f\subseteq^*\bigcup \mathcal{F}_f\cup\bigcup_{i\leq n_0}A_i$, a contradiction.
\end{proof}

Examining the proof of Theorem~\ref{mainthm1} one observes that in fact we only needed $\mathfrak{s}_{mm} < \mathfrak{c}$ alongside the existence of countably infinitely many elements $A_n \in \mathcal A$ whose associated complemented filters are not ultrafilters to obtain the contradiction. It follows that the same proof can be used to derive the following proposition which complements Corollary~\ref{ultraf}. 

\begin{prp}
If $\mathcal{A}$ is a maximal ideal independent family of size  $<\mathfrak{c}$, then there are at most finitely many $A\in\mathcal{A}$ for which the corresponding complemented filter is not an ultrafilter. 
\end{prp}

Note, this is not necessarily the case for maximal ideal independent families with cardinality $\mathfrak{c}$, as any completely separable maximal almost disjoint  family $\mathcal{A}$ is maximal ideal independent (see Corollary~\ref{csmad}), and for any $A\in\mathcal{A}$ the corresponding complemented filter $\mathcal{F}(\mathcal{A},A)$ is not an ultrafilter, as it is generated by cofinite subsets of $A$. 

\section{Arbitrarily Large  Maximal Ideal Independent Families}

In this section we examine the question of how to adjoin via forcing a maximal ideal independent family of desired cardinality and thus begin an investigation of the spectrum of such families. The {\em spectrum} of maximal ideal independent families, denoted ${\rm spec}(\mathfrak{s}_{mm})$, is defined as the set of all cardinalities of maximal ideal independent families. Throughout $V$ denotes the ground model and $\mathbb{C}$ denotes the poset for adding a single Cohen real.

\begin{lemma}\label{extension_lemma}
Let $\mathcal{A}$ be an ideal independent family. There is a $ccc$ forcing $\mathbb{P}(\mathcal{A})$ which adds a set $z$ such that in $V^{\mathbb{P}(\mathcal{A})}$:
\begin{enumerate}
\item $\mathcal{A}\cup\{z\}$ is an ideal independent family, and 
\item for each $y\in V\cap ([\omega]^\omega\backslash\mathcal{A})$ the family $\mathscr{A}\cup\{z,y\}$ is not ideal independent.
\end{enumerate}
\end{lemma}

\begin{proof}
Add a Cohen real to $V$ and consider a filter $\mathcal{U}$ which contains the Cohen real and is maximal with respect to the following property: 

\medskip
$(\ast)$ For any $X\in\mathcal{U}$, any $A\in\mathcal{A}$ and finite $\mathcal F\subseteq \mathcal{A}\setminus\{A\}$, the set $X\cap (A\setminus\bigcup \mathcal F)$ is infinite.  
\medskip

Let  $\mathbb{M}(\mathcal{U})$ be Mathias forcing relativized to $\mathcal{U}$, let $x$ be the generic real added by $\mathbb{M}(\mathcal{U})$ over $V^{\mathbb{C}}$, let  $\dot{x}$ be $\mathbb{M}(\mathcal{U})$-name for $x$ (in $V^{\mathbb{C}}$) and let  $\mathbb{P}(\mathcal{A})=\mathbb{C}*\mathbb{M}(\dot{\mathcal{U}})$. 


\begin{claim}
In $V^{\mathbb{P}(\mathcal{A})}$ the family $\mathcal{A}\cup\{\omega\setminus x\}$ is ideal independent.
\end{claim}
\begin{proof}
Let $\mathcal F$ be a finite subset of $\mathcal{A}$. First we prove that $\bigcup \mathcal F$ does not almost contain $\omega\setminus x$.  
Let $A\in\mathcal{A}\backslash \mathcal F$, $(s,B)\in\mathbb{M}(\mathcal{U})$ and let $n\in\omega$ be arbitrary. Since the Cohen real belongs to $\mathcal{U}$, we can assume that $B$ is a subset of it, so $A\setminus(B\cup \bigcup \mathcal F)$ is infinite. Let $k\in\omega$ be big enough so $[\max(s),k)\cap (A\setminus(B\cup \bigcup \mathcal F))$ has more than $n$ elements. Then $(s\cup\{k\}, B)$ forces that $(\omega\setminus\dot{x})\setminus\bigcup \mathcal F$ has more than $n$ elements and since $n$ was arbitrary, it follows that $(\omega\setminus x)\setminus\bigcup \mathcal F$ is infinite. A genericity argument shows that $x\cap (A\setminus\bigcup \mathcal F)$ is infinite for any $A\in\mathcal{A}$ and $\mathcal F\in[\mathcal{A}\setminus\{A\}]^{<\omega}$, which implies that $(A\setminus\bigcup \mathcal F)\setminus(\omega\setminus x)$ is infinite.
\end{proof}

\begin{claim}
Let $A\in ([\omega]^\omega\cap V)\backslash\mathcal{A}$. Then in $V^{\mathbb{P}(\mathcal{A})}$, $\mathcal{A}\cup\{\omega\setminus x, A\}$ is not ideal independent. 
\end{claim}

\begin{proof}
Let $A\in[\omega]^\omega\cap V$ be an arbitrary set. If there are $X\in\mathcal{U}$, $B\in\mathcal{A}$ and $\mathcal F\in[\mathcal{A}\setminus\{B\}]^{<\omega}$ such that $X\cap(B\setminus\bigcup \mathcal F)\subseteq^* A$, then $ x\cap(B\setminus\bigcup \mathcal F)\subseteq^* A$. But $ x\cap(B\setminus\bigcup \mathcal F)=(B\setminus\bigcup \mathcal F)\setminus(\omega\setminus x)$, so $A$ can not be added to $\mathcal{A}\cup\{\omega\setminus x\}$, as witnessed by $B,\mathcal F$ and $\omega\setminus x$. On the other hand, if for all $X\in\mathcal{U}$, $B\in\mathcal{A}$ and $F\in[\mathcal{A}\setminus\{B\}]^{<\omega}$ it happens that $X\cap(B\setminus \bigcup F)\nsubseteq^* A$, then $(\omega\setminus A)\cap X\cap(B\setminus \bigcup F)$ is infinite. Thus, by maximality of $\mathcal{U}$, $\omega\setminus A\in\mathcal{U}$, which implies that $x\subseteq^*\omega\setminus A$ and so $A\subseteq^*\omega\setminus x$. 
\end{proof}
This completes the proof of the Lemma.
\end{proof}


\begin{thm}\label{mainthm2}
Assume $GCH$. Let $C$ be a set of uncountable cardinals and let $\kappa$ be a regular uncountable cardinal such that $\sup C\leq \kappa$. Then there is a $ccc$ generic extension in which $$C\subseteq{\rm spec}(\mathfrak{s}_{mm}).$$
\end{thm}

\begin{proof}
Add $\kappa$ Cohen reals to the ground model $V$ to obtain a model of  $\mathfrak{c}=\kappa$ and for each $\lambda\in C$ let $\mathcal{A}_\lambda$ be an ideal independent family of cardinality $\lambda$. Let $\langle\lambda_\beta:\beta<\gamma\rangle$ be an enumeration of $C$. Proceed with a finite support iteration $\langle\mathbb{P}_\alpha,\dot{\mathbb{Q}}_\alpha:\alpha<\omega_1\rangle$ where each iterand is a finite support iteration of length, the cardinality of $C$, as follows:

Let $\mathbb{P}_0$ be the finite support iteration $\langle\mathbb{R}_\beta^0,\dot{\mathbb{S}}_j^0:j<\gamma\rangle$ defined by $\mathbb{R}_0^0=\mathbb{P}(\mathcal{A}_{\lambda_0})$ and $\mathbb{R}_j^0\Vdash\dot{\mathbb{S}}_j^0=\mathbb{P}(\mathcal{A}_{\lambda_j})$. After forcing with $\mathbb{P}_0$, for each $j<\gamma$, define $\mathcal{A}_j^0=\mathcal{A}_{\lambda_j}\cup\{x_j\}$, where $x_j$ is the real from Lemma \ref{extension_lemma} added by the $j$-th-step of the iteration $\mathbb{P}_0$.
Now, assume $\mathbb{P}_\beta$ and $\{\mathcal{A}_j^\beta:j<\gamma\}$ are defined. The next step $\dot{\mathbb{Q}}_\beta$ is the finite support iteration $\langle \mathbb{R}_j^\beta,\dot{\mathbb{S}}_j^\beta:j<\gamma\rangle$ such that $\mathbb{R}_0^\beta=\mathbb{P}(\mathcal{A}_{0}^\beta)$ and $\mathbb{R}_j^\beta\Vdash\dot{\mathbb{S}}_j^\beta=
\mathbb{P}(\mathcal{A}_{j}^\beta)$. In $V[G_{\beta+1}]$, after forcing with $\mathbb{P}_{\beta}*\dot{\mathbb{Q}}_\beta$, define $\mathcal{A}_{j}^{\beta+1}=\mathcal{A}_{j}^{\beta}\cup\{x_j\}$, where $x_j$ is the real from Lemma \ref{extension_lemma} added by the $j$-th-step of the iteration $\dot{\mathbb{Q}}_\beta$. If $\beta$ is a limit ordinal and $\mathbb{P}_\alpha$, $\mathcal{A}_j^\alpha$ are defined for all $\alpha<\beta$ and $j<\gamma$, let $\mathbb{P}_\beta$ be the finite support iteration $\langle\mathbb{P}_\alpha,\dot{\mathbb{Q}}_\alpha:\alpha<\beta\rangle$ and for $j<\gamma$, $\mathcal{A}_j^\beta=\bigcup_{\alpha<\beta}\mathcal{I}_j^\alpha$.

Let $\mathbb{P}_{\omega_1}$ be the above iteration and for any $j<\gamma$, let $\mathcal{A}_j^*=\bigcup_{\alpha<\omega_1}\mathcal{A}_j^\alpha$. Fix $j$. Let  $y\in V[G_{\omega_1}]\cap [\omega]^\omega\backslash\mathcal{A}_j^*$. Then, there is $\alpha<\omega_1$ such that $y$ is added at stage $\alpha$ of the iteration. Since $y\notin\mathcal{A}_j$, it is also the case that $y\notin\mathcal{A}_j^{\alpha+1}$ and so by the diagonalization properties of the generic $x$ added at by $\mathbb{S}_j^\alpha$, $\mathcal{A}^j_{\alpha+1}\cup\{x,y\}$ is not ideal independent. Since $x\in \mathcal{A}_j$, $\mathcal{A}_j\cup{y}$ is not ideal independent. Thus $\mathcal{A}_j^*$ is maximal. The poset  $\mathbb{P}_{\omega_1}$ preserves all cardinals and we only added $\omega_1$ sets to the family $\mathcal{A}_j$ to obtain $\mathcal{A}_j^*$, $\mathcal{A}_j^*$ has indeed size $\lambda_j$.
\end{proof}

\begin{remark}
The cardinality of a maximal ideal inde\-pen\-dent family can have countable cofinality, while the character of any ultrafilter is uncountable. The first assertion follows from the previous theorem by taking $\kappa>\aleph_\omega$. The second assertion can be found in \cite{ultrafilters_on_omega_cardinal_characteristics}.
\end{remark}

\section{Forcing Invariant Maximal Ideal Independent Families}

In the following, we construct a maximal ideal independent family with strong combinatorial properties, which guarantee that its maximality is preserved by a large number of forcing notions.

\begin{dfn}\label{encompassing}
Let $\mathcal U$ be an ultrafilter. A maximal ideal independent family $\mathcal{A}$ is called $\mathcal{U}\hbox{-\em encompassing}$ if the following conditions hold:
\begin{enumerate}
    \item $\mathcal{U} \cap \mathcal{A} = \emptyset$, i.e. $\mathcal{A}$ is contained in the dual ideal of $\mathcal{U}$.
    \item For every $X \in \mathcal{U}$ the set of $A \in \mathcal{A}$ so that $X \in \mathcal{F}(\mathcal{A}, A)$ is co-countable. 
\end{enumerate}
\end{dfn}

\begin{thm}\label{mainthm3}
Assume $\mathsf{CH}$. For any $p$-point $\mathcal{U}$ there is a $\mathcal{U}$-encompassing maximal ideal independent family $\mathcal{A}$ such that for all $A\in\mathcal{A}$, the corresponding complemented filter $\mathcal{F}(\mathcal{A},A)$ is a $p$-point.
\end{thm}

\begin{proof}
Let $\mathcal{U}$ be a $p$-point and $\langle Y_\alpha:\alpha\in\omega_1\rangle$ be an $\subseteq^*$-decreasing sequence which generates the $p$-point $\mathcal{U}$. Let $\langle X_\alpha:\alpha\in\omega_1\rangle$ be an enumeration of $[\omega]^\omega$. By recursion we construct a sequence $\langle\mathcal{A}_{\alpha}:\alpha\in[\omega,\omega_1)\rangle$ such that:
\begin{enumerate}
    \item For all $\alpha$, $\mathcal{A}_\alpha\subseteq\mathcal{U}^*$ is a countable ideal independent family.
    \item For all $\alpha$, if $X_\alpha \notin \mathcal{A}_{\alpha+1}$ then $\mathcal{A}_{\alpha+1}\cup\{X_\alpha\}$ is not an ideal independent family.
    \item For all $\alpha$, $\mathcal{A}_{\alpha+1}=\mathcal{A}_\alpha$ or $\mathcal{A}_{\alpha+1}=\mathcal{A}_\alpha\cup\{A_0^\alpha,A_1^\alpha\}$ for some $A_0^\alpha,A_1^\alpha\in\mathcal{U}^*$ and such that $A_0^\alpha\setminus A_1^\alpha,A_1^\alpha\setminus A_0^\alpha\subseteq Y_\alpha$.
    \item If $\alpha$ is a limit ordinal, then $\mathcal{A}_\alpha=\bigcup_{\beta<\alpha}\mathcal{A}_\beta$.
    \item If $A_i^\alpha\in\mathcal{A}$ is added in step $\alpha$ of the iteration, $\langle \mathcal{P}^{\alpha,i}_\beta:\beta\in[\alpha+1,\omega_1)\rangle$ is the enumeration of all partitions of $A^\alpha_i$, for $i\in 2$, and for all $\alpha$ and $\beta>\alpha$, there are a finite $F\subseteq[\beta\setminus\{\alpha\}]^{<\omega}$ and $k\in\omega$, such that for any $i,j\in 2$ the set $A^\alpha_i\setminus\left(A_j^\beta\cup \bigcup F\right)\setminus k$:
    \begin{itemize}
    \item either is a partial selector of the partition $\mathcal{P}^{\alpha,i}_\beta$, 
    
    \item or is contained in one element of the partition $\mathcal{P}^{\alpha,i}_\beta$.
    \end{itemize}
\end{enumerate}
After the recursion we define $\mathcal{A}=\bigcup_{\alpha<\omega_1}\mathcal{A}_\alpha$. Condition (1) makes sure that $\mathcal{A}$ is an ideal independent family and (2) makes sure that $\mathcal{A}$ is maximal. Condition (3) makes sure that $\mathcal{A}$ is $\mathcal{U}$-encompassing. Condition (5) makes sure that the filters $\mathcal{F}(\mathcal{A},A)$ are selective ultrafilters for all $A\in\mathcal{A}$. We start by setting $\mathcal{A}_\omega=\langle A_n:n\in\omega\rangle$ be a partition of $\omega$ into infinitely many infinite sets. For each $n\in\omega$, let $A_0^n=A_1^n=A_n$ and let $\langle \mathcal{P}_\beta^{n,i}:\beta\in[\omega,\omega_1)\rangle$  enumerate all partitions of $A_n$. Assume $\mathcal{A}_\alpha$ has been constructed. We take care of the set $X_\alpha$ and define $\mathcal{A}_{\alpha+1}$.

If {\emph{$X_\alpha\in\mathcal{U}$}}, we just define $\mathcal{A}_{\alpha+1}=\mathcal{A}_\alpha$, and condition (2) from Definition \ref{encompassing} will make sure that $X_\alpha$ can not be added to the family $\mathcal{A}$.  If {\emph{$X_\alpha$ is in the ideal generated by the family $\mathcal{A}_\alpha$}} we have nothing to do and we can define $\mathcal{A}_{\alpha+1}=\mathcal{A}_\alpha$ again. {\emph{Otherwise}}, $X_\alpha\notin\mathcal{U}$ and $X_\alpha$ is positive relative to the ideal generated by $\mathcal{A}_\alpha$. Let $e_\alpha:\omega\to\mathcal{A}_\alpha$ be an enumeration of the elements of $\mathcal{A}_\alpha$, and define $C_0=e_\alpha(0), C_{n+1}=e_\alpha(n+1)\setminus\bigcup_{i\leq n}e_\alpha(i)$.

\medskip
If there are $n\in\omega$ and finite $F\subseteq\mathcal{A}_\alpha\setminus \{e_\alpha(0),\ldots,e_\alpha(n)\}$ such that $C_n\setminus X_\alpha\subseteq^*\bigcup F$ and $C_n\cap X_\alpha$ is infinite, then we have $C_n\setminus \bigcup F\subseteq^*C_n\cap X_\alpha\subseteq X_\alpha$, and we can define again $\mathcal{A}_{\alpha+1}=\mathcal{A}_\alpha$. 

\medskip
So let us assume that for all $n\in\omega$, $C_n\setminus X_\alpha$ is finite or $C_n\setminus X_\alpha$ is not covered by any $F\subseteq\mathcal{A}_\alpha\setminus \{e_\alpha(0),\dots,e_\alpha(n)\}$. Since $e_\alpha(n)$ is not almost contained in the union of finitely many elements from $\mathcal{A}_\alpha\setminus\{e_\alpha(n)\}$, and $\mathcal{A}_\alpha\setminus\{e_\alpha(n)\}$ is countable, by recursion we can construct an infinite set $B_n\subseteq e_\alpha(n)$, such that:
\begin{enumerate}
    \item $B_n\cap X_\alpha=\emptyset$, 
    \item for all $Z\in\mathcal{A}_\alpha$ different from $e_\alpha(n)$, we have $Z\cap B_n=^*\emptyset$,     
    \item  and moreover, by going to a subset if necessary, $B_n$ is a partial selector of the partition $\mathcal{P}^{\gamma,i}_\alpha$ or is completely contained in one element of the partition $\mathcal{P}^{\gamma,i}_\alpha$, where $\gamma$ and $i$ are such that $A_i^\gamma=e_\alpha(n)$    
    \end{enumerate}

Also, since for all $n\in\omega$, $C_n\notin \mathcal{U}$ and $\mathcal{U}$ is a $p$-point,  there is $A\in\mathcal{U}$ such that for all $n\in\omega$, $C_n\cap A$ is finite, $X_\alpha\cap A=\emptyset$ and $A\subseteq Y_\alpha$. Let $W_0,W_1$ be infinite disjoint subsets of $A$ which are not in the ultrafilter $\mathcal{U}$. Define $W$ as,
\begin{equation*}
    W=\left(\bigcup_{n\in\omega} C_n\setminus B_n\right)\setminus A
\end{equation*}
Finally, define $A_i^\alpha=W\cup W_i$, for $i\in 2$, and $\mathcal{A}_{\alpha+1}=\mathcal{A}_\alpha\cup\{A_0^\alpha,A_1^\alpha\}$. Note that $X_\alpha\subseteq A_i^\alpha$, $A_i^\alpha\setminus A_{1-i}^\alpha=W_i\subseteq Y_\alpha$. It remains to observe that by construction of $A_0^\alpha$ and $A_1^\alpha$ the family  $\mathcal{A}_{\alpha+1}$ is ideal independent.
\end{proof}

\begin{thm}\label{preserving_encompassing}
Let $\mathcal{U}$ be a $p$-point and let $\mathbb P$ be a proper, $\om^\om$-bounding forcing notion which preserves $p$-points. Then $\mathbb P$ preserves the maximality of any $\mathcal{U}$-encompassing maximal ideal independent family $\mathcal{A}$ such that for all $A\in\mathcal{A}$, the corresponding complemented filter $\mathcal{F}(\mathcal{A},A)$ is a $p$-point.
\end{thm}

Note that this theorem implies that under $\mathsf{CH}$, in the generic extension by any proper $\om^\om$-bounding $p$-point preserving forcing notion $\mathfrak{s}_{mm}$ is $\aleph_1$.

\begin{proof}
Fix an ultrafilter $\mathcal{U}$, a $\mathcal{U}$-encompassing maximal ideal independent family $\mathcal{A}$ with the property that all of the complemented filters of $\mathcal{A}$ are $p$-points, and a proper, $\om^\om$-bounding, $p$-point preserving forcing notion $\mathbb P$. Let $p \in \mathbb P$ and let $\dot{X}$ be a name so that $p \forces \dot{X} \in [\omega]^\omega$. We need to show that some $q \leq p$ forces that $\dot{X}$ cannot be added to $\mathcal{A}$ without destroying ideal independence. More precisely, this means that we need to either find a $q \leq p$ so that $q$ forces that $\dot{X}$ is in the ideal generated by $\mathcal{A}$ or else find $q \leq p$ and an $A \in \mathcal{A}$ so that $q$ forces that $\dot{X}$ is in the complemented filter corresponding to $A$. 

Thus suppose towards a contradiction that $p$ forces that $\dot{X}$ is neither in the ideal generated by $\mathcal{A}$ nor in any filter $\mathcal{F}(\mathcal{A}, A)$ for any $A \in \mathcal{A}$. Note that this implies in particular that $\dot{X}$ is not in $\mathcal{U}$ since if it were, then in would be in some filter $\mathcal{F}(\mathcal{A}, A)$ (in fact co-countably many). Since $\mathbb P$ preserves $\mathcal{U}$ being a $p$-point it follows that $p$ forces that the complement of $\dot{X}$ is in $\mathcal{U}$ and therefore we can find a $q \leq p$ and a $Z \in \mathcal{U}$ so that $q \forces \dot{X} \cap Z = \emptyset$. Fix such a $q$ and $Z$. To complete the proof it suffices to therefore show that some $r \leq q$ forces that $n \in \dot{X}$ for some $n \in Z$. 

For any $u \in \mathbb P$ let $X_u = \{n\in \omega \; | \; u \nVdash \check{n} \notin \dot{X}\}$ be the {\em outer hull} of $\dot{X}$ with respect to $u$, i.e. the collection of $n < \omega$ forced to be in $\dot{X}$ by some $u' \leq u$. Note that $u \forces \dot{X} \subseteq \check{X}_u$ for any $u \in \mathbb P$. It follows that for any condition $u$ stronger than $p$, $X_u$ is not in the ideal generated by $\mathcal{A}$. By the maximality of $\mathcal{A}$, moreover we get that for every $r \leq q$ the set $X_r$ is in some complemented filter of $\mathcal{A}$. Therefore to finish the proof it suffices to show that in fact any such $X_r$ is actually in uncountably many such filters. This suffices since if this is the case then in particular it applies to $X_q$ and, since, by the definition of $\mathcal{U}$-encompassing, $Z$ is in $\mathcal{F}(\mathcal{A}, A)$ for co-countably many $A \in \mathcal{A}$, there is some $A \in \mathcal{A}$ so that $Z \cap X_q \in \mathcal{F}(\mathcal{A}, A)$ and so $Z \cap X_q$ has infinite intersection. Thus, some $r \leq q$ forces that $n \in \dot{X}$ for some (in fact infinitely many) $n \in Z$. Summing up, it suffices to show the following claim.

\begin{claim}
For any $u \in \mathbb P$ stronger than $p$ the set $X_u \in \mathcal{F}(\mathcal{A}, A)$  for uncountably many $A \in \mathcal{A}$.
\end{claim}

Fix such a $u \in \mathbb P$ and suppose towards a contradiction that there were only countably many $A \in \mathcal{A}$ with $X_u \in \mathcal{F}(\mathcal{A}, A)$. Let $M \prec H_\theta$ be a countable model for $\theta$ sufficiently large with $\mathbb{P}, p, u, \mathcal{A}, \mathcal{U} \in M$ containing every $A$ so that $X_u \in \mathcal{F}(\mathcal{A}, A)$. Enumerate $\mathcal{A} \cap M$ as $\{A_i \;  | \; i < \omega\}$. Since $p$, and hence $u$, forces that for each $n < \omega$ the name $\dot{X}$ is not in $\mathcal{F}(\mathcal{A}, A_n)$, and each one of such filters is a $p$-point by our assumption and hence an ultrafilter preserved by $\mathbb P$, there is in $M$ a dense set of conditions below $u$ forcing that $\dot{X} \cap A_n \setminus \bigcup_{i < k_n, i \neq n} A_i$ is finite for some $k_n \in \omega$. Applying $\om^\om$-boundedness and properness we can find in the ground model functions $f, g \in \om^\om$ and a condition $r \leq u$ which is $(M, \mathbb{P})$-generic so that for each $n < \omega$ we have 

$$r \forces \dot{X} \cap A_n \setminus \bigcup_{i < \check{f}(n), i \neq n} A_i \subseteq \check{g}(n)$$

In particular we get that $X_r \cap A_n \setminus \bigcup_{i < f(n), i \neq n} A_i \subseteq g(n)$ and thus $X_r \notin \mathcal{F}(\mathcal{A}, A_n)$ for any $n < \omega$. But then, by applying the same argument to $r$ that we applied to $u$, we get that $X_r$ is in some $\mathcal{F}(\mathcal{A}, B)$ for some $B \in \mathcal{A}$ with $B \neq A_n$ for any $n < \omega$. This is a contradiction however since $X_r \subseteq X_u$ and by definition of the $A_n$'s $X_u \notin \mathcal{F}(\mathcal{A}, B)$.  This contradiction implies that $X_q$ is in uncountably many complemented filters of $\mathcal{A}$ and hence the proof is complete.
\end{proof}

As an straightforward corollary we obtain:
\begin{crl}\hfill
\begin{enumerate}
    \item $\mathfrak{s}_{mm} = \aleph_1$ in the Sacks model.
    \item $\mathfrak{s}_{mm} = \aleph_1$ in the Miller partition model and hence $\mathfrak{s}_{mm} < \mathfrak{a}_T$ is consistent.
    \item $\mathfrak{s}_{mm} = \aleph_1$ in the $h$-perfect tree forcing model and hence $\mathfrak{s}_{mm} < \hbox{non}(\mathcal{N})$ 
is consistent.
\end{enumerate}
\end{crl}

\begin{proof}
For (1), it is a standard fact that the iterated Sacks forcing preserves $p$-points and it is $\omega^\omega$-bounding. For (2), in \cite{miller_partition_forcing}, Miller has constructed a forcing, known as Miller partition forcing,  which makes the cardinal invariant $\mathfrak{a}_T$ equal to $\aleph_2$, as recently shown in \cite{JCVFOGJS} preserves $p$-points, and as shown in 
\cite{spinas_partition_miller}) is $\omega^\omega$-bounding. For (3) recall, that the $h$-perfect tree forcing is proper, $^\omega\omega$-bounding, preserves $p$-points and that in the $h$-perfect tree forcing model $\hbox{non}(\mathcal{N})=\aleph_2$, see \cite[Section 2]{goldstern-judah-shelah}. 
\end{proof}

An alternation of Miller partition and $h$-perfect tree forcings will lead to a model of $\mathfrak{i}=\mathfrak{s}_{mm}<\hbox{non}(\mathcal{N})=\mathfrak{a}_T=\aleph_2$
(for the effect of the respective posets on $\mathfrak{i}$ see \cite{JCVFOGJS} and \cite{CS}).

\begin{crl}
$\mathfrak{s}_{mm}$ is independent of $\mathfrak{a}_T$
\end{crl}

\begin{proof}
In the Miller partition model, $\mathfrak{s}_{mm} < \mathfrak{a}_T$. On the other hand, it is well known that $\mathfrak{a}_T < \mathfrak{u}$ holds in the Random model and hence $\mathfrak{a}_T < \mathfrak{s}_{mm}$ holds in that model as well.
\end{proof}

\section{Ideal Mr\'owka Spaces}\label{idealMrowka}

To each ideal independent family $\A$ we can associate a topological space $\psi(\A)$ defined as follows.

\begin{definition*}
Let $\A$ be an ideal independent family. The space $\psi(\A)$ is the set $\omega\cup\A$ endowed with the topology $\tau$ which satisfies the following condition:
\begin{itemize}
    \item a set $U\in \tau$ if and only if for any $A\in \A\cap U$ there exists $F\in \F(\A,A)$ such that $F\subseteq U$. 
\end{itemize}
\end{definition*}

It is easy to see that the set $\omega$ is a dense discrete subset of $\psi(\A)$ and for any $A\in\A$ the family $\{F\cup\{A\}:F\in\F(\A,A)\}$ forms an open neighborhood base at $A$ in $\psi(\A)$. The set $\A\subseteq \psi(\A)$ is closed and discrete. For any $A\in\A$ the set $\{A\}\cup A$ is a clopen neighborhood of $A$. Thus, we obtain the following:

\begin{lemma}\label{topol}
For any ideal independent family $\A$ the space $\psi(\A)$ is Tychonoff zero-dimensional separable scattered (of height $2$) space.
\end{lemma}

Observe that for any almost disjoint family $\A$ the space $\psi(\A)$ is the well-known Mr\'{o}wka space corresponding to the almost disjoint family $\A$.


\begin{lemma}
Let $\A$ be an ideal independent family. If $\A$ is maximal, then each $C\in \A^+$ has an accumulation point in $\psi(\A)$. Moreover, if all but finitely many $\F(\A,A)$, $A\in \A$  are ultrafilters, then $\A$ is maximal, whenever  each $C\in \A^+$ has an accumulation point.
\end{lemma}

\begin{proof}
Let $\A$ be a maximal ideal independent family and $C\in \A^+$. The maximality of $\A$ implies that there exists $A\in \A$ such that $C\in\F(\A,A)$. Then $A$ is an accumulation point of the set $C$.

Let $\A$ be an ideal independent family such that $\F(\A,A)$ is an ultrafilter for all but finitely many $A\in \A$, and  each $C\in \A^+$ has an accumulation point in $\psi(\A)$. Denote $\mathcal P=\{A\in \A: \F(\A,A)$ is not an ultrafilter$\}$.  Consider any $C\in \A^+$. Let $D=C\setminus \bigcup\mathcal P$. Since the family $\mathcal P$ is finite, the set $D$ is positive. Then there exists an accumulation point $A\in \A$ of $D$ in $\psi(\A)$. It follows that $D\cap F\neq \emptyset$ for any $F\in\F(\A,A)$. Clearly, $A\notin \mathcal P$. Then $\F(\A,A)$ is an ultrafilter and, consequently, $C\supseteq D\in \F(\A,A)$. Thus, the ideal independent family $\A$ is maximal. 
\end{proof}

For an ideal independent family $\A$ by $\mathsf{Cl}_{\A}(\w)$ we denote the set of all closed subsets of $\w$ in $\psi(\A)$. 
The following lemma topologically describes the ideal $\J(\A)$ corresponding to an ideal independent family $\A$.

\begin{lemma}\label{cl}
For any ideal independent family $\A$ we have that $\J(\A)\subseteq \mathsf{Cl}_{\A}(\w)$. Moreover, if $\A$ is maximal, then $\J(\A)= \mathsf{Cl}_{\A}(\w)$.
\end{lemma}
\begin{proof}
Let $A_1, A_2$ be any distinct elements of an ideal independent family $\A$ and put $C=A_1\cap A_2$. Observe that for any $A\in\A\setminus\{A_2\}$ the set $A\setminus A_2\in \F(\A,A)$ and $(A\setminus A_2)\cap C=\emptyset$. Similarly, the set $A_2\setminus A_1\in\F(\A,A_2)$ and $(A_2\setminus A_1)\cap C=\emptyset$. Thus, the set $C$ is closed and discrete in $\psi(\A)$. At this point it is easy to see that for any $A\in\A$ the ideal $\J(\A)$ consists of closed (in $\psi(\A)$) sets. Since a finite union of closed sets is closed and any subset of a closed discrete set is closed, we get that $\J(\A)\subseteq \mathsf{Cl}_{\A}(\w)$.

Let $\A$ be a maximal ideal independent family. Consider an arbitrary set $C\in \mathsf{Cl}_{\A}(\w)$. Since the set $C$ is closed we get that $C\notin \F(\A,A)$ for every $A\in \A$. The maximality of $\A$ implies that $C\notin\A^+$. Hence there exists a finite collection $\{A_i:i\in n\}\in [\A]^{<\w}$ such that $C\subseteq^*\bigcup_{i\in n}A_i$. Since $C$ is closed, for every $i\in n$ there exists $F_i\in \F(\A,A_i)$ such that $F_i\cap C=\emptyset$. Then for every $i\in n$ there exists $m(i)\in\w$ and a finite family $\{D^i_j:j\in m(i)\}\in [\A\setminus\{A_i\}]^{<\w}$ such that $(A_i\setminus \bigcup_{j\in m(i)}D^i_j)\cap C=^*\emptyset$, witnessing that $C\subseteq^*\bigcup_{i\in n}\bigcup_{j\in m(i)}(A_i\cap D_j^i)\in \J(\A)$. Hence $\J(\A)= \mathsf{Cl}_{\A}(\w)$. 
\end{proof}

A Tychonoff space $X$ is called {\em pseudocompact} if each continuous real-valued function on $X$ is bounded.
The following lemma is a folklore.

\begin{lemma}\label{folk}
Let $X$ be a space with an open discrete dense subspace $D$. Then $X$ is pseudocompact if and only if each $A\in [D]^{\omega}$ has an accumulation point in $X$.
\end{lemma}

\begin{lemma}\label{comp}
Let $\A$ be an ideal independent family. The space $\psi(\A)$ is pseudocompact if and only if $\A$ is a maximal almost disjoint family.
\end{lemma}

\begin{proof}
If $\A$ is a maximal almost disjoint family, then by~\cite{hrusak-separable-mad} the Mr\'{o}wka space $\psi(\A)$ is pseudo\-compact. 

Assume that the ideal independent family $\A$ is not almost disjoint. Then there exists $A_1,A_2\in \A$ such that the set $C=A_1\cap A_2$ is infinite. By Lemma~\ref{cl}, the set $C$ is closed in $\psi(\A)$. Lemma~\ref{folk} implies that the space $\psi(\A)$ is not pseudocompact.  
\end{proof}

Corollary~\ref{csmad} and Lemma~\ref{comp} imply the following. 
\begin{prp}\label{top1}
For a maximal ideal independent family $\A$ the following conditions are equivalent:
\begin{enumerate}
    \item $\psi(\A)$ is pseudocompact;
    \item $\A$ is a maximal almost disjoint family;
    \item $\A$ is a completely separable maximal almost disjoint family.
\end{enumerate}
\end{prp}

\begin{prp}\label{top2}
Let $\A$ be an ideal independent family. The space $\psi(\A)$ is locally compact if and only if $\psi(\A)$ is a homeomorphic to $\psi(\B)$ for some almost disjoint family $\B$.
\end{prp}

\begin{proof}
The sufficiency follows from the fact that the Mr\'{o}wka space $\psi(\B)$ is locally compact for any almost disjoint family $\B$.

Assume that $\psi(\A)$ is locally compact for some ideal independent family $\A$. Since the space $\psi(\A)$ is zero-dimensional (see Lemma~\ref{topol}), for each $A\in \A$ we can find a compact open neighborhood $U_A$ of $A$. Recall that for every $A\in \A$ the set $A\cup\{A\}$ is a clopen neighborhood of $A$. Then $V_A=U_A\cap (A\cup\{A\})$ is a compact open neighborhood of $A$. Since $V_A\cap \A=\{A\}$ we deduce that $A$ is the only non-isolated point of $V_A$. The compactness of $V_A$ implies that $B_A=V_A\cap \w$ is a convergent sequence to $A$ in $\psi(\A)$. Let $\B=\{B_A:A\in \A\}$. Clearly, the family $\B$ is almost disjoint. Define $f:\psi (\B)\rightarrow \psi(\A)$ by 
$$f(x)=
\begin{cases}
n &\hbox{ if }  x=n\in \w;\\
A &\mbox{ if } x=B_A\in \B.
\end{cases}
$$
A routine verification shows that the map $f$ is a homeomorphism.
\end{proof}

\section{Conclusion and Open Questions}

Of interest remains the following question:

\begin{question}
Is it consistent that for every maximal ideal independent family $\A$ there is an $A\in\A$ so that $\mathcal{F}(\A, A)$ is an ultrafilter?
\end{question}

A positive answer to this question would imply that there are no completely separable maximal almost disjoint families. It is known that such families exist under either $\mathfrak{s} \leq \mathfrak{a}$ or $2^{\aleph_0} < \aleph_\omega$, see \cite{heike-dilip-juris} and \cite{hrusak-separable-mad}, respectively. 

In the context of topology we ask more about Mr\'{o}wka spaces for ideal independent families. In this paper we have barely scratched the surface of what may be possible, for instance the following is a natural question which is open.

\begin{question}
Does there exist a $\ZFC$ example of an infinite maximal  ideal  independent family $\A$ such that $\psi(\A)$ is Fr\'{e}chet-Urysohn or first-countable?
\end{question}

Clearly, the above question has positive answer if completely separable maximal almost disjoint families exist in ZFC.

The results of the current paper together with those of \cite{cancino_guzman_miller_2021} give either a $\ZFC$ relation, or establish the independence between $\mathfrak{s}_{mm}$ and any other well studied cardinal characteristic, with the exception of the almost disjointness number $\mathfrak{a}$. The following remains open.

\begin{question}
Is it consistent that $\mathfrak{s}_{mm} < \mathfrak{a}$?
\end{question}

The corresponding question for $\mathfrak{i}$, i.e. the consistency of $\mathfrak{i}<\mathfrak{a}$ is one of the most interesting open problems in cardinal characteristics of the continuum and many of the roadblocks towards solving that problem are the same as trying to answer the question above. See the appendix of \cite{JCVFOGJS} for an interesting discussion on Vaughan's problem. 

As noted in the introduction, Theorem \ref{mainthm3} implies that $\mathfrak{s}_{mm}={\rm max}\{\mathfrak{d}, \mathfrak{u}\}$ in many standard forcing extensions. However, this is not a $\ZFC$ equality as $\mathfrak{s}_{mm}>{\rm max}\{\mathfrak{d}, \mathfrak{u}\}$ holds in the Boolean ultrapower model, see for example \cite{Brendle_templates1}. That model requires a measurable cardinal and increases both $\mathfrak{u}$ and $\mathfrak{d}$. As a result the following two questions remain very interesting:
\begin{question}
Is ${\rm max}\{\mathfrak{d}, \mathfrak{u}\}<\mathfrak{s}_{mm}$ consistent with $\ZFC$?
\end{question}

\begin{question}
If $\mathfrak{d} = \mathfrak{u} = \aleph_1$ does $\mathfrak{s}_{mm} = \aleph_1$?
\end{question}
The later question is an ideal independent version of Roitman's problem. Theorem \ref{mainthm2} opens up the possibility of a maximal ideal independent families of size $\aleph_\omega$. We can therefore ask:

\begin{question}
Is it consistent that $\mathfrak{s}_{mm} = \aleph_\omega$? More generally can $\mathfrak{s}_{mm}$ have countable cofinality?
\end{question}

Finally, we can ask more generally about the spectrum of maximal ideal independent families:
\begin{question}
What $\ZFC$ restrictions are there on the set ${\rm spec}(\mathfrak{s}_{mm})$? Can it be equal to any set of regular cardinals which includes the continuum?
\end{question}

\end{document}